\documentclass{amsart}

\usepackage{amsmath}
\usepackage[psamsfonts]{amssymb}
\usepackage[latin1]{inputenc}
\usepackage{pstricks,pst-node,pst-plot}

\usepackage{hyperref}
\usepackage{diagbox}

\usepackage{graphicx}
\usepackage{graphics}
\usepackage{pgfplots}
\tikzset{ font={\fontsize{9pt}{12}\selectfont}}

\newtheorem{theorem}{\bf Theorem}
\newtheorem{proposition}[theorem]{\bf Proposition}
\newtheorem{lemma}[theorem]{\bf Lemma}
\newtheorem{corollary}[theorem]{\bf Corollary}

\theoremstyle{remark}

\def\cal{\mathcal}
\def\ssm{{\smallsetminus}}
\def\and{{\quad\text{and}\quad}}
\def\with{{\quad\text{with}\quad}}

\def\eps{{\varepsilon}}

\def\bs{\boldsymbol}

\def\N{{\mathbb N}}
\def\Z{{\mathbb Z}}

\def\C{{\mathbb C}}

\def\R{{\mathbb R}}
\def\TT{{\mathbb T}}

\def\Sp{{\widehat{\TT}}}

\def\qf{{\cal Q}_f}

\def\D{{\mathrm D}}
\def\T{{\rm T}}
\def\d{{\rm d}}
\def\dz{{{\rm d}z}}
\def\dw{{{\rm d}w}}

\def\re{{\rm Re}}
\def\im{{\rm Im}}
\def\deg{{\rm deg}}
\def\card{{\rm card}}
\def\ord{{\rm ord}}
\def\id{{\rm id}}
\def\res{{\rm residue}}

\def\Quad{{\cal Q}}
\def\vf{{{\cal V}_f}}
\def\cf{{{\cal C}_f}}
\def\tf{{{\mathcal T}_f}}
\def\qf{{\Quad_f}}
\def\sf{{{\cal S}_f}}

\thanks{Supported in part by the ANR grant Lambda ANR-13-BS01-0002 and by FRPDF allotment 2018-19 of Presidency University}

\begin{author}[K.~Banerjee]{Kuntal Banerjee}
\email{kbanerjee.maths@presiuniv.ac.in}
\address{ %
  Presidency University\\
86/1 College Street, Kolkata - 700073\\
West Bengal\\
India  }
\end{author}

\begin{author}[X.~Buff]{Xavier Buff}
\email{xavier.buff@math.univ-toulouse.fr}
\address{ %
  Institut de Mathématiques de Toulouse\\
   UMR5219\\ Université de Toulouse, CNRS, UPS\\ F-31062 Toulouse Cedex 9\\ France }
\end{author}

\begin{author}[J.~Canela]{Jordi Canela}
\email{jordi.canela@unir.net}
\address{ %
  Escuela Superior de Ingenier\'ia y Tecnolog\'ia\\
 Universidad Internacional de la Rioja\\
  Av.\ de la Paz, 137  \\
  26006 Logro\~no \\
  Spain }
\end{author}

\begin{author}[A. Epstein]{Adam Epstein}
\email{a.l.epstein@warwick.ac.uk}
\address{ %
  Mathematics Institute\\
University of Warwick\\ Coventry CV4 7AL - UK}
\end{author}

\title{Tips of Tongues in the Double Standard Family}

\begin{document}

\begin{abstract}
We answer a question raised by Misiurewicz and Rodrigues concerning the family of degree 2 circle maps $F_\lambda:\R/\Z\to \R/\Z$ defined by 
\[F_\lambda(x) := 2x + a+ \frac{b}{\pi} \sin(2\pi x)\with \lambda:=(a,b)\in \R/\Z\times (0,1).\] We prove that if $F_\lambda^{\circ n}-{\rm id}$ has a zero of multiplicity $3$ in $\R/\Z$, then there is a system of local coordinates $(\alpha,\beta):W\to \R^2$ defined in a neighborhood $W$ of $\lambda$, such that $\alpha(\lambda) =\beta(\lambda)=0$ and $F_\mu^{\circ n} - {\rm id}$ has a multiple zero with $\mu\in W$ if and only if $\beta^3(\mu) = \alpha^2(\mu)$. This shows that the tips of  tongues are regular cusps. 
\end{abstract}

\maketitle

\section*{Introduction}

Following Misiurewicz and Rodrigues  \cite{mr1},  we consider the family of circle maps $F_\lambda:\R/\Z\to \R/\Z$ defined by 
\[F_\lambda(x) := 2x + a+\frac{b}{\pi} \sin(2\pi x)\quad \text{with}\quad \lambda:= (a,b)\in \R/\Z\times [0,1].\] 
If  $b\in [0,1/2)$, then $F_\lambda:\R/\Z\to \R/\Z$ is expanding and all periodic cycles of $F_\lambda$ in $\R/\Z$ are repelling. 
If $b\in [1/2,1]$, it may happen that $F_\lambda:\R/\Z\to \R/\Z$ has a non-repelling cycle. 
The multiplier of such a cycle belongs to $[0,1]$. There is at most one such cycle. 
Connected components of the open sets of parameters  $\lambda\in(a,b)\in \R/\Z\times [0,1]$ for which $F_{\lambda}$ has an attracting cycle are called {\em tongues} (see \cite{mr1} and \cite{D}). The period of the attracting cycle remains constant in each tongue, and is called the period of the tongue 

Let $T$ be a tongue of period $p\geq 1$. 
The boundary of $T$ consists of two smooth curves which are graphs with respect to $b$ and intersect tangentially at the tip $\lambda_T\in \R/\Z\times (0,1)$ (see \cite{mr1, mr2} and Figure~\ref{fig:tongues}). If $\lambda\in\partial T$ then $F_{\lambda}$ has a cycle of period $p$ and multiplier 1. On the one hand, if $\lambda\in\partial T\ssm \{\lambda_T\}$, then points of the cycle are double zeros of $F_\lambda^{\circ p}-\id$. On the other hand, points of the cycle are triple zeros of $F_{\lambda_T}^{\circ p}-\id$.

\begin{figure}[hbt!]
 \begin{tikzpicture}
    \begin{axis}[width=300pt, axis equal image, scale only axis,  enlargelimits=false, axis on top,
    xtick={0,0.25,0.5,0.75,1}, ytick={0.5, 0.75, 1}]
      \addplot graphics[xmin=0,xmax=1,ymin=0.5,ymax=1] {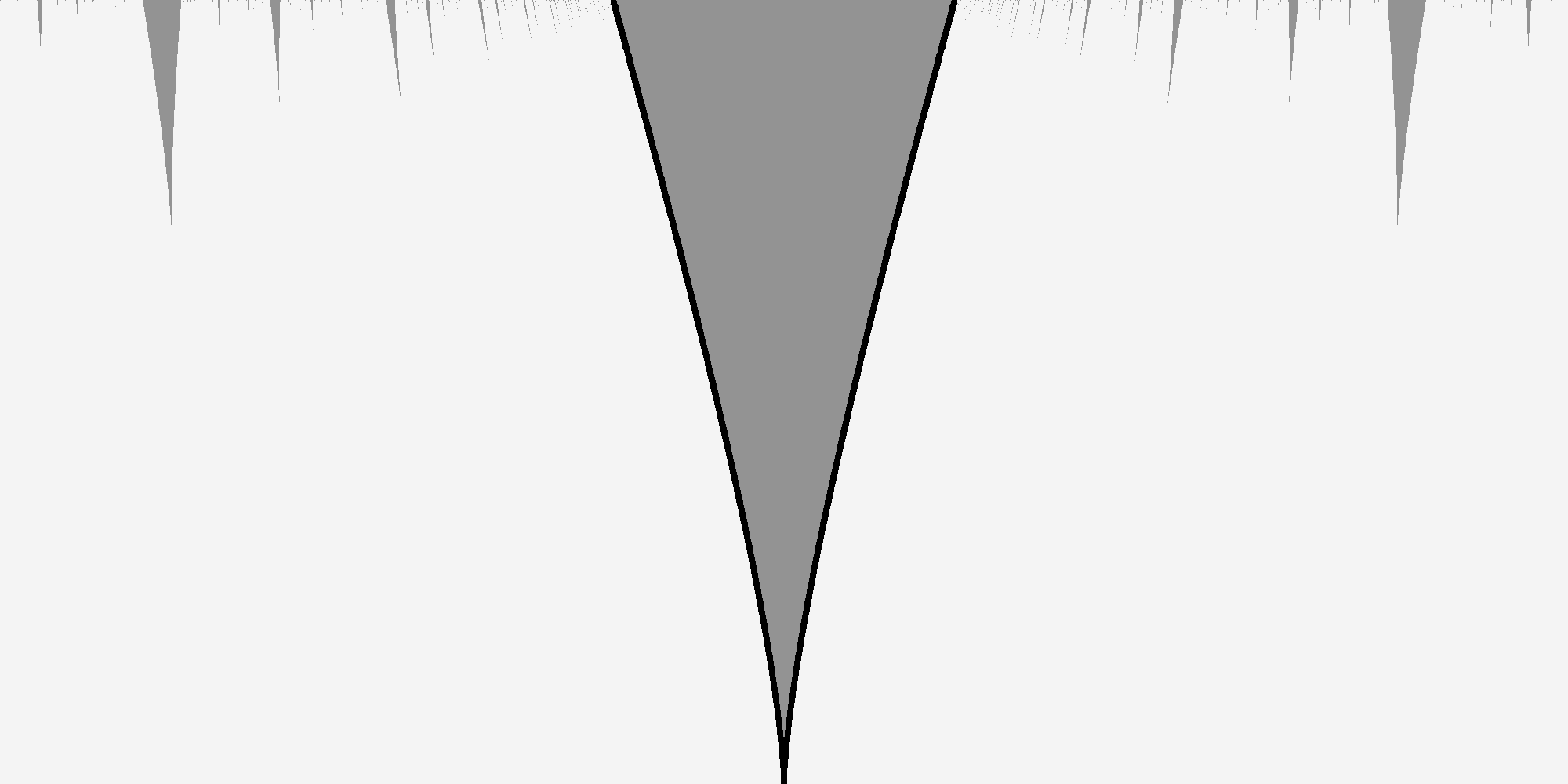};
    \end{axis}
  \end{tikzpicture}
\caption{The tongues of the family $F_{\lambda}$. The horizontal axis corresponds to the parameter $a$ and the vertical axis to $b$. We draw in black the boundary of the tongue of period 1. \label{fig:tongues}}
\end{figure}
 
 There is a unique tongue of period $1$. Misiurewicz and Rodrigues \cite{mr1} proved that the order of contact of its two boundary curves 
 at the tip is $1/2$. In  \cite{mr2} they asked whether this property holds for all tongues of the family $F_{\lambda}$. In this article, we answer positively to this question. More precisely, we prove that near the tip of any tongue, the two boundary curves form an ordinary cusp. 
 
%
%
%

\begin{theorem}\label{theo:main}
Assume $F_\lambda^{\circ n}-{\rm id}$ has a zero of multiplicity $3$ in $\R/\Z$. Then there is a system of local coordinates $(\alpha,\beta):W\to \R^2$ defined a neighborhood $W$ of $\lambda$ in $\R/\Z\times (0,1)$, such that $\alpha(\lambda) =\beta(\lambda) = 0$  and $F_\mu^{\circ n} - {\rm id}$ has a multiple zero with $\mu\in W$ if and only if $\beta^3(\mu) = \alpha^2(\mu)$. 
\end{theorem}

Our proof relies on a transversality result due to Adam Epstein for families of finite type analytic maps, which itself relies on an injectivity result of a linear map acting on an appropriate space of quadratic differentials.  In \S\ref{sec:finitetype}, we prove that the maps $F_\lambda$ are finite type analytic maps. In \S\ref{sec:coords}, we define the functions $\alpha$ and $\beta$. In \S\ref{sec:der}, we identify the derivatives of those functions at $\lambda$. 
In \S\ref{sec:inj}, we state and prove the injectivity result. 
In \S\ref{sec:trans}, we prove that $(\alpha,\beta)$ is a system of local coordinates. 

Some classical results on quadratic differentials are collected in Appendix \ref{sec:qd}. 

\section*{Notation}

If $U$ is a complex manifold, we denote by $\T U$ the tangent bundle of $U$ and for $z\in U$, we denote by $\T_z U$ the tangent space to $U$ at $z$. 
If $\phi:U\to \C$ is a holomorphic function, we denote by $\d\phi:\T U\to \C$ the exterior derivative of $\phi$ (this is a holomorphic $1$-form on $U$). 
If $F:U\to V$ is a holomorphic map between complex manifolds $U$ and $V$, we denote by  $\D F:\T U\to \T V$ the bundle map $\T_z U\ni v\mapsto \D_zF(v)\in \T_{F(z)}V$. 

Assume $f:U\to V$ is a holomorphic map between Riemann surfaces. If $\omega$ is a holomorphic $1$-form on $V$, then $f^*\omega:= \omega \circ \D f$ is a holomorphic $1$-form on $U$. If $\vartheta$ is a holomorphic vector field on $U$, then there is a meromorphic vector field $f^*\vartheta$ on $U$ satisfying $\D f\circ f^*\vartheta = \vartheta\circ f$.

We will consider various holomorphic families $t\mapsto \gamma_t$ defined near $0$ in $\C$. We will employ the notation 
\[\gamma:=\gamma_0\quad\text{and}\quad \dot\gamma:=\frac{d\gamma_t}{dt}\Big|_{t=0}.\]

\section{Finite type analytic maps\label{sec:finitetype}}

The notion of finite type analytic maps originates in \cite{E}. 
Let $f:{\mathbb X}\rightarrow {\mathbb Y}$ be an analytic map of complex 1-manifolds, possibly disconnected. 
An open set $V\subseteq {\mathbb Y}$ is {\em evenly covered} by $f$ if $f_{|U}:U\rightarrow V$
is a homeomorphism for each component $U$ of $f^{-1}(V)$; we say that $y\in
{\mathbb Y}$ is a {\em regular value} for $f$ if some neighborhood $V\ni y$ is evenly
covered, and a {\em singular value} for $f$ otherwise. Note that the set 
$\sf$ of singular values is closed. Recall that $x\in {\mathbb X}$ is a {\em critical
point} if the derivative of $f$ at $x$ vanishes, and then $f(x)\in {\mathbb Y}$ is a 
{\em critical value}. We say that $y\in {\mathbb Y}$ is an {\em asymptotic value} if $f$
approaches $y$ along some path tending to infinity relative to ${\mathbb X}$. It
follows from elementary covering space theory that the critical values
together with the asymptotic values form a dense subset of $\sf$. In
particular, every isolated point of $\sf$ is a critical or asymptotic value.

An analytic map $f:{\mathbb X}\rightarrow {\mathbb Y}$ of complex $1$-manifolds is of
{\em finite type} if 
\begin{itemize}
\item $f$ is nowhere locally constant, 
\item $f$ has no isolated removable singularities, 
\item ${\mathbb Y}$ is a finite union of compact Riemann surfaces, and 
\item $\sf$ is finite. 
\end{itemize}

If ${\mathbb Y}$ is connected, we define $\deg f$ as the finite or infinite
number ${\rm \card}\bigl(f^{-1}(y)\bigr)$ which is independent of $y\in {\mathbb Y}\ssm \sf$.
When $f:{\mathbb X}\to {\mathbb Y}$ is a finite type analytic map with ${\mathbb X}\subseteq {\mathbb Y}$, we say that $f$ is a finite type analytic map on ${\mathbb Y}$. 

We first prove that the maps $F_\lambda$ extend to finite type analytic maps. 

\subsection{Preliminaries}

Set $\TT:=\C/\Z$ and  $\Lambda:=\TT\times \C^*$. Let $F:\Lambda\times \TT\to \TT$ be the holomorphic map defined by 
\[F(\lambda,z) = 2z + a + \frac{b}{\pi} \sin(2\pi z)\quad \text{with}\quad \lambda:=(a,b)\in \Lambda.\] 
For $\lambda \in \Lambda$, let $F_\lambda:\TT\to \TT$ be the holomorphic map defined by 
\[F_\lambda(z) := F(\lambda,z).\]

It will be convenient to consider the global coordinate $\TT\ni z\mapsto w:={\rm e}^{2\pi {\rm i}z}\in \C^*$. Note that $w:\TT\to \C^*$ is an isomorphism. Thus, adding two points denoted 
$z=+{\rm i}\infty$ (or $w=0$) and $z=-{\rm i}\infty$ (or $w=\infty$), $\TT$ may be compactified into a Riemann surface 
$\Sp$ isomorphic to the Riemann sphere. 

We will prove that for $\lambda\in \Lambda$, the map $F_\lambda:\TT\to \Sp$ is a finite type analytic map on $\Sp$. 

\subsection{The singular set}

Fix $\lambda := (a,b)\in \Lambda$ and set $f:=F_{a,b}:\TT\to \Sp$. 
Note that 
\[w\circ f= {\rm e}^{2\pi {\rm i} a} w^2 {\rm e}^{b(w-1/w)}\]
and
\[f^*(\dw) = {\rm e}^{2\pi {\rm i} a} {\rm e}^{b(w-1/w)}(bw^2+2w+b)\ \dw.\]
In particular, $f$ has two critical points counting multiplicities: the solutions of $bw^2+2w+b=0$, i.e., the points $c^±\in \T$ such that 
\[w(c^±) = \frac{-1±\sqrt{1-b^2}}{b}.\] 
If $b\neq 1$, those are simple critical points of $f$. We denote by $\cf\:=\{c^+,c^-\}\subset \TT$ the set of critical points of $f$ and
by $\vf:=f(\cf)\subset \TT$ the set of critical values of $f$.

\begin{lemma}
The singular set $\sf$ is equal to $\vf\cup \{±{\rm i}\infty\}$. 
\end{lemma}

\begin{proof}
We already identified the set of critical values of $f$. Note that $±{\rm i}\infty$ are singular values since those points are omitted values. It is therefore enough to show that $f$ does not have any asymptotic value in $\TT$. 

If $v\in\TT$ is an asymptotic value, then there exists a curve $\gamma:[0,1)\rightarrow \TT$, such that $\gamma(t)\to ±{\rm i}\infty$ and $f\circ \gamma(t)\to v$ as $t\to 1$. We assume that $\gamma(t)\rightarrow + {\rm i}\infty$. The proof for the case $\gamma(t)\rightarrow - {\rm i}\infty$ is analogous.  

It is convenient to lift via the canonical covering $\bs \pi:\C\to \TT:=\C/\Z$. Choose $A\in \C$  such that $\bs \pi(A) = a$. Let $\tilde f:\C\to \C$ be defined by 
\[ \tilde f(Z)=2Z+ A+\frac{b}{\pi}\sin(2\pi Z)\quad \text{so that}\quad \bs \pi \circ \tilde f = f\circ \bs\pi.\]
Let $\Gamma:[0,1)\to \C$ be a lift of $\gamma:[0,1)\to \TT$, i.e., satisfying $\bs \pi \circ \Gamma = \gamma$. 
Then, $\tilde f\circ \Gamma$ is a lift of $f\circ \gamma$, thus $\tilde f\circ \Gamma(t)$ converges in $\C$ as $t\to 1$. 

Set $X:=\re(\Gamma):[0,1)\to \R$ and $Y:=\im(\Gamma):[0,1)\to \R$. Then, 
\[\tilde f\circ \Gamma=2(X+{\rm i}Y)+A+\frac{b}{\pi}\sin(2\pi\Gamma),\quad  \sin(2\pi\Gamma)=\frac{{\rm e}^{-2\pi Y}{\rm e}^{2\pi {\rm i}X}-{\rm e}^{2\pi Y}{\rm e}^{-2\pi {\rm i}X}}{2{\rm i}}\]
and $Y(t)\to +\infty$ as $t\to 1$. 
It follows that as $t\to 1$,  
\[\tilde f\circ \Gamma(t)\sim 2X(t)  -\frac{b}{4\pi {\rm i}}{\rm e}^{2\pi Y(t)}{\rm e}^{-2\pi {\rm i}X(t)}.\]
We can distinguish 2 cases.
 If there exists a sequence $\{t_k\}_{k\in\N}$ converging to $1$  with $\bigl\{X(t_k)\bigr\}_{k\in \N}$ bounded, then 
 \[\bigl|\tilde f\circ \Gamma(t_k)\bigr|\sim \frac{b}{4\pi}{\rm e}^{2\pi Y(t_k)}\underset{k\to +\infty}\longrightarrow +\infty.\]
Otherwise, $X(t)\rightarrow \pm\infty$ as  $t\rightarrow 1$ and there exists a sequence $\{t_k\}_{k\in\N}$ converging to $1$ with $X(t_k)\in \Z$ for all $k\in \N$, so that 
\[\tilde f\circ \Gamma(t_k) \sim 2X(t_k)+{\rm i}\frac{b}{4\pi} {\rm e}^{2\pi Y(t_k)}\underset{k\to +\infty}\longrightarrow \infty.\]
In both cases,  the sequence $\bigl\{\tilde f\circ \Gamma(t_k)\bigr\}_{k\in \N}$ cannot converge in $\C$. 
\end{proof}

\begin{corollary}\label{coro:covering}
The map $f:\TT\to \Sp$ is a finite type analytic map on $\Sp$. More precisely, $f : \TT\ssm f^{-1}(\vf)\to \TT\ssm \vf$ is a covering map. 
\end{corollary}

\section{Splitting triple zeros\label{sec:coords}}

In the remainder of the article, we fix a parameter $\lambda:=(a,b)\in \R/\Z\times (0,1)$ such that $F_{\lambda}^{\circ n}-\id$ has a triple zero $x\in \R/\Z$. 
We set $f:=F_{\lambda}:\TT\to \TT$. The point $x$ is  periodic for $f$ with period $p$ dividing $n$. For $k\geq 0$, we set $x_k:= f^{\circ k}(x)$ and we denote by $\left<x\right>:=\{x_0,x_1,\ldots,x_{p-1}\}$ the cycle of $x$. 

Since $f:\R/\Z\to \R/\Z$ preserves the orientation, the multiplier of $f^{\circ p}$ at $x$ is necessarily $1$ and there is a local coordinate 
$\zeta:(\TT,x)\to (\C,0)$ vanishing at $x$ satisfying 
\begin{equation}\label{eq:zeta}
\zeta(\bar z) = \bar \zeta(z) \quad\text{and}\quad \zeta\circ f^{\circ p}  = \zeta + \zeta^3 + {\cal O}(\zeta^5).
\end{equation}
According to the Weierstrass Preparation Theorem, there exist a neighborhood $W_1\subset \Lambda$ of $\lambda$, a neighborhood $W_2\subset \C$ of $0$ and analytic functions $A:W_1\to \C$, $B:W_1\to \C$, $C:W_1\to \C$ and $g:W_1\times W_2\to \C$ such that for $\mu\in W_1$, 
\begin{equation}\label{eq:Weierstrass}
\zeta\circ F_\mu^{\circ p} - \zeta = P_\mu(\zeta) \cdot g(\mu,\zeta)
\end{equation}
with 
\begin{equation}\label{eq:ABC}
A(\lambda) = B(\lambda) = C(\lambda)=0,\quad g(\lambda,\zeta) = 1 + {\cal O}(\zeta^2)
\end{equation}
and
\begin{equation}\label{eq:P}
P_\mu(\zeta) := A(\mu) + B(\mu) \zeta + C(\mu) \zeta^2 + \zeta^3.
\end{equation}
The polynomial $P_{\lambda}$ has a zero of multiplicity $3$ at $0$, and as $\mu$ varies in $W_1$, this zero splits in
three zeros (counting multiplicities) of $P_\mu$. When $\mu\in \R/\Z\times (0,1)$, the map $F_\mu^{\circ n} - \id$ commutes with 
$z\mapsto \bar z$, so that the polynomial $P_\mu$ has real coefficients. For such a parameter $\mu$, a multiple zero of $P_\mu$ is necessarily real. 

For any $\mu\in \Lambda$, the function $\zeta\circ F_\mu^{\circ p} - \zeta$ vanishes at the periodic points of $F_\mu$ of period dividing $p$, and so, divides $\zeta\circ F_\mu^{\circ n} - \zeta$ which vanishes at the periodic points of period dividing $n$. In addition, if $n=mp$, then 
$\zeta\circ f^{\circ n} - \zeta = m \zeta^3 + {\cal O}(\zeta^5)$. 
So, there is an analytic function $h:W_1\times W_2\to \C$ such that  for $\mu\in W_1$, 
\[\zeta\circ F_\mu^{\circ n} - \zeta = P_\mu(\zeta) \cdot h(\mu,\zeta)\quad \text{with}\quad h(\lambda,\zeta) = m+{\cal O}(\zeta^2)
\]
Since $f$ only has two critical points in $\TT$, it has a single non-repelling cycle, that is, the cycle $\left<x\right>$.  
All other cycles of $f$ in $\R/\Z$ are repelling. Shrinking $W_1$ is necessary, it follows that  for 
$\mu\in W_1$, the function $\zeta\circ F_\mu^{\circ n}-\zeta$ has a multiple zero in $\R/\Z$ if and only if the polynomial 
$P_\mu$ has a multiple zero in $\R/\Z$. According to the previous discussion, this is the case if and only if $P_\mu$ has a multiple zero.  

Let $\alpha : W_1\to \C$ and $\beta:W_1\to \C$ be defined by 
\[\alpha :=   \frac{C^3}{27}- \frac{BC}{6}+\frac{A}{2}\quad \text{and}\quad \beta :=  \frac{C^2}{9}-\frac{B}{3}.\]
Then, 
\[{\rm discriminant}(P_\mu) = 108\beta^3(\mu)-108\alpha^2(\mu).\]
So, if $\mu\in W_1$, the polynomial $P_\mu$ has a multiple zero if and only if $\beta^3(\mu)=\alpha^2(\mu)$. 

In order to prove Theorem \ref{theo:main}, it is therefore enough to show that $(\alpha,\beta)$ is a system of local coordinates near $\lambda$. For this purpose, we shall show that the restrictions of $\d \alpha $ and $\d \beta$ to ${\rm T}_{\lambda}\Lambda$ are linearly independent. 
Since $A$, $B$ and $C$ vanish at $\lambda$, 
\[\d \alpha |_{{\rm T}_{\lambda}\Lambda}= \frac{1}{2} \d A|_{{\rm T}_{\lambda}\Lambda}\quad \text{and}\quad 
\d \beta|_{{\rm T}_{\lambda}\Lambda} = -\frac{1}{3} \d B|_{{\rm T}_{\lambda}\Lambda}.\]
It is therefore enough to show that the forms $\d A|_{{\rm T}_{\lambda}\Lambda}$ and $\d B|_{{\rm T}_{\lambda}\Lambda}$ are linearly independent.

\section{Identifying the derivatives\label{sec:der}}

Here, we identify $\d A(v)$ and $\d B(v)$ for $v\in \T_\lambda\Lambda$.
First, to each $v\in \T_\lambda\Lambda$, we shall associate a meromorphic vector field $\vartheta_v$ on $\TT$ having simple poles along $\cf\cup \{±{\rm i}\infty\}$,  such that for all $z\in \TT\ssm \cf$, 
\[\D f\circ \vartheta_v(z) := \D_{\lambda,z}F(v,0).\]
Second, for $k\in [1,p]$, let $\zeta_k:(\TT,x_k)\to (\C,0)$ be the local coordinate vanishing at $x_k$ defined by 
\[\zeta_k := \zeta\circ f^{\circ (p-k)}.\]
Our identification goes as follows. 

\begin{proposition}\label{prop:dAdB}
Let $q_A$ and $q_B$ be quadratic differentials, defined and meromorphic near $\left<x\right>$, such that $q_A-(\d\zeta_k)^2/\zeta_k$ and $q_B - (\d\zeta_k)^2/\zeta_k^2$ are holomorphic at $x_k$ for all $k\in [1,p]$. 
Then, for all $v\in \T_\lambda\Lambda$, 
\[\d A(v) = \sum_{k=1}^p {\rm residue}(q_A\otimes \vartheta_v,x_k)\quad \text{and}\quad \d B(v) =  \sum_{k=1}^p {\rm residue}(q_B\otimes \vartheta_v,x_k).\]
\end{proposition}

 In the remaining parts of this section we prove Proposition~\ref{prop:dAdB}.
 
\subsection{Meromorphic vector fields\label{sec:mvf}}

Assume $v\in \T_\lambda\Lambda$ and $z\in \TT\ssm \cf$. Then, the derivative $\D_z f : \T_z\TT\to \T_{f(z)}\TT$ is an isomorphism and $\D_{\lambda,z}F(v,0)\in \T_{f(z)}\TT$. Let $\vartheta_v$ be the vector field defined on $\TT\ssm \cf$ by 
\[\vartheta_v(z) := (\D_z f)^{-1}\bigl(\D_{\lambda,z}F(v,0)\bigr)\in \T_z\TT.\]

\begin{lemma}
For all $v\in \T_\lambda \Lambda$, the vector field $\vartheta_v$ is holomorphic on $\TT\ssm \cf$, meromorphic on $\Sp$, vanishes at $z=±{\rm i}\infty$ and has at worst simple poles along $\cf$. 
\end{lemma}

\begin{proof}
The map $v\mapsto \vartheta_v$ is linear. So, it is enough to prove the result for $v_a:=\d/\d a$ and $v_b:=\d /\d b$. We have 
\[\vartheta_{v_a} = \frac{2\pi {\rm i} {\rm e}^{2\pi {\rm i} a} w^2 {\rm e}^{b(w-1/w)}}{{\rm e}^{2\pi {\rm i} a} {\rm e}^{b(w-1/w)}(bw^2+2w+b) }\frac{\d}{\dw} = \frac{2\pi {\rm i} w^2}{bw^2+2w+b}\frac{\d}{\dw}\]
and
\[\vartheta_{v_b} = \frac{{\rm e}^{2\pi {\rm i} a} w^2 (w-1/w){\rm e}^{b(w-1/w)}}{{\rm e}^{2\pi {\rm i} a} {\rm e}^{b(w-1/w)}(bw^2+2w+b) }\frac{\d}{\dw} = \frac{w^3-w}{bw^2+2w+b}\frac{\d}{\dw}\]
Those two vector fields have the required properties. 
\end{proof}

Denote by $\tf$ the space of meromorphic vector fields on $\Sp$ which are holomorphic on $\TT\ssm \cf$, vanish at $±{\rm i}\infty$ and have at worst simple poles along $\cf$. In other words, 
\[\tf := \left\{\frac{c_3 w^3 + c_2 w^2 + c_1 w}{bw^2+2w+b}\frac{\d}{\dw}~;~ (c_1,c_2,c_3)\in \C^3\right\}.\]
Let $\Theta_f:\T_\lambda \Lambda \to \tf$ be the linear map defined by 
\[\Theta_f(v) := \vartheta_v.\]
Let $\tau\in \tf$ be the radial vector field
\[\tau := w\frac{\d}{\dw}.\]
Note that $\tau - f^*\tau$ belongs to $\tf$. Indeed, 
\[\tau-f^*\tau = \frac{bw^3+w^2+bw}{bw^2+2w+b}\frac{\d}{\dw}\in \tf.\]

\begin{lemma}\label{lem:tf}
The space $\tf$ is the direct sum of the image of $\Theta_f$ and the line spanned by $\tau-f^*\tau$:
\[\tf = {\rm Im}(\Theta_f) \oplus {\rm Vect}(\tau-f^*\tau).\]
\end{lemma}

\begin{proof}
The dimension of $\tf$ is $3$. Thus, it is enough to show that the three vector fields $\vartheta_{v_a}$, $\vartheta_{v_b}$ and $\tau-f^*\tau$ are linearly independent. Equivalently, it is enough to show that the three functions 
\[w^2, \quad w^3-w\quad \text{and}\quad bw^3+w^2+bw\]
are linearly independent. This is true since $b\neq 0$. 
\end{proof}

Assume now $v\in \T_\lambda\Lambda$ and let $t\mapsto \lambda_t\in \Lambda$ be a curve such that $\dot\lambda = v$. Let $t\mapsto f_t$ be the family of maps defined by 
\[f_t:= F_{\lambda_t} :\TT\to \TT.\]
Then, for each $z\in \TT$,
\[\dot f(z)  = D_{\lambda,z}F(v,0) = \D_zf\circ \vartheta_v(z)\quad \text{with}\quad \vartheta_v := \Theta_f(v)\in \tf.\]
%

\begin{lemma}\label{lem:dftn}
For all $k\geq 1$, 
\[\frac{\d f_t^{\circ k}}{\d t}\Big|_{t=0} = \D f^{\circ k} \circ \vartheta_v^k\quad\text{with}\quad \vartheta_v^k:=\vartheta_v+f^*\vartheta_v + \cdots + \bigl(f^{\circ (k-1)}\bigr)^*\vartheta_v.\]
\end{lemma}

\begin{proof}
The proof follows from an elementary induction on $k\geq 1$ using the following fact: if $h_t = g_t\circ f_t$ with $\dot f = \D f\circ \vartheta$ and $\dot g = \D g \circ \tau$, then 
\[\dot h = \dot g\circ f + \D g\circ \dot f = \D g \circ \tau\circ f + \D g\circ \D f \circ \vartheta = \D h \circ (f^*\tau + \vartheta).\qedhere\]
\end{proof}

Note that the poles of $\vartheta_v^n$ are the critical points of $f$ and their iterated preimages (up to order $n-1$). The two critical points of $f$ are in $\TT\ssm\R/\Z$, and so are all their preimages. Therefore, $\vartheta_v^n$ is holomorphic in a neighborhood of $\R/\Z$. In particular, it is holomorphic near the parabolic periodic point $x\in\R/\Z$. 

\subsection{Polar parts of quadratic differentials}

Our identification of the derivatives $\d A|_{\T_\lambda\Lambda}$ and $\d B|_{\T_\lambda\Lambda}$ relies on the use of quadratic differentials (see Appendix \ref{sec:qd} for basics regarding quadratic differentials). Recall that $\zeta:(\TT,x)\to (\C,0)$ is a local coordinate vanishing at $x$ such that 
\[\zeta\circ f^{\circ p} = \zeta + \zeta^3 + {\cal O}(\zeta^5).\]
We shall use the quadratic differential $(\d\zeta)^2/\zeta$ and $(\d\zeta)^2/\zeta^2$ which are defined and meromorphic near $x$ in $\TT$.

Following \S\ref{subsec:pairing}, if $Z\subset \T$ is a finite set, if $q$ is a quadratic differential, defined and meromorphic near $Z$ and if $\vartheta$ is a vector field, defined and meromorphic near $Z$, we shall use the notation 
\[\left< q,\vartheta\right>_Z := \sum_{z\in Z} {\rm residue}(q\otimes \vartheta,z).\]
If $q$ has at worst simple poles along $Z$ and if $\theta$ is defined on $Z$ with $\theta(z)\in \T_z\TT$ for $z\in Z$, we shall use the notation 
\[\left< q,\theta\right>_Z := \left< q,\vartheta\right>_Z\]
where $\vartheta$ is any vector field, defined and holomorphic near $Z$, with $\vartheta(z) = \theta(z)$ for $z\in Z$. The result does not depend on the choice of extension. 

\begin{lemma}
For all $v\in \T_\lambda\Lambda$, 
\[\d A(v)  = \left<\frac{(\d \zeta)^2}{\zeta},\vartheta_v^p\right>_x\quad \text{and}\quad \d B(v)  = \left<\frac{(\d \zeta)^2}{\zeta^2},\vartheta_v^p\right>_x.\]
\end{lemma}

\begin{proof}
According to Equations \eqref{eq:Weierstrass}, \eqref{eq:ABC} and  \eqref{eq:P}, 
\[\zeta\circ f_t^{\circ p} - \zeta = \bigl(A(\lambda_t) + B(\lambda_t)\zeta + {\cal O}(\zeta^2)\bigr)\cdot \bigl(1+{\cal O}(\zeta^2)\bigr).\]
Taking the derivative with respect to $t$ and evaluating at $t=0$ yields
\[\d\zeta\circ \D f^{\circ p} \circ \vartheta_v^p = \d A(v) + \d B(v)\zeta + {\cal O}(\zeta^2).\]
According to Equation \eqref{eq:zeta}, 
\[\zeta\circ f^{\circ p} = \zeta + {\cal O}(\zeta^3)\quad \text{so that}\quad \d\zeta\circ \D f^{\circ p} = \bigl(1+{\cal O}(\zeta^2)\bigr)\d \zeta.\]
As a consequence
\[\d\zeta(\vartheta_v^p) = \d\zeta\circ \D f^{\circ p} \circ \vartheta_v^p + {\cal O}(\zeta^2)= \d A(v) + \d B(v)\zeta + {\cal O}(\zeta^2).\]
Thus, 
\[\d A(v)  = {\rm residue}\left(\frac{\d\zeta(\vartheta_v^p)}{\zeta}\d\zeta,x\right) = \left<\frac{(\d \zeta)^2}{\zeta},\vartheta_v^p\right>_x\]
and similarly
\[\d B(v)  = {\rm residue}\left(\frac{\d\zeta(\vartheta_v^p)}{\zeta^2}\d\zeta,x\right) = \left<\frac{(\d \zeta)^2}{\zeta^2},\vartheta_v^p\right>_x.\qedhere\]
\end{proof}

Rather than working near $x$ with the vector field $\vartheta_v^p$, it will be convenient to work along the cycle $\left<x\right>$ with the vector field $\vartheta_v$. Recall that for $k\in [1,p]$, the local coordinate $\zeta_k:(\TT,x_k)\to (\C,0)$ vanishes at $x_k$ and is defined by 
\[\zeta_k := \zeta\circ f^{\circ (p-k)}.\]

\begin{lemma}\label{lem:invariance}
For all $k\in \Z/p\Z$, 
\[f^*\left(\frac{(\d\zeta_{k+1})^2}{\zeta_{k+1}}\right) - \frac{(\d\zeta_{k})^2}{\zeta_k}\quad\text{and}\quad f^*\left(\frac{(\d\zeta_{k+1})^2}{\zeta_{k+1}}\right) - \frac{(\d\zeta_{k})^2}{\zeta_k}\]
are holomorphic near $x_k$. 
\end{lemma}

\begin{proof}
If $k\in [1,p-1]$, then $\zeta_k = \zeta_{k+1}\circ f$, so that 
\[f^*\left(\frac{(\d\zeta_{k+1})^2}{\zeta_{k+1}}\right) = \frac{(\d\zeta_{k})^2}{\zeta_k}\quad \text{and}\quad
f^*\left(\frac{(\d\zeta_{k+1})^2}{\zeta_{k+1}^2}\right) = \frac{(\d\zeta_{k})^2}{\zeta_k^2}.\]
If $k=p$, then $\zeta_p = \zeta$ and $\zeta_1 \circ f= \zeta\circ f^{\circ p} = \bigl(1+{\cal O}(\zeta_p^2)\bigr)\zeta_p$. As a consequence, 
$f^*(\d \zeta_1) = \bigl(1+{\cal O}(\zeta_p^2)\bigr)\d \zeta_p$, 
\[f^*\left(\frac{(\d\zeta_1)^2}{\zeta_1}\right) = \bigl(1+{\cal O}(\zeta_p^2)\bigr)\frac{ (\d\zeta_p)^2}{\zeta_p} 
\quad \text{and}\quad f^*\left(\frac{(\d\zeta_1)^2}{\zeta_1^2}\right) = \bigl(1+{\cal O}(\zeta_p^2)\bigr)\frac{ (\d\zeta_p)^2}{\zeta_p^2} .\qedhere\]
\end{proof}


\begin{proof}[Proof of Proposition \ref{prop:dAdB}]
Recall that $\zeta_p = \zeta$.
According to the previous lemma, for all $k\in \Z/p\Z$, 
\[(f^{\circ k})^* \frac{(\d\zeta_k)^2}{\zeta_k} - \frac{(\d\zeta)^2}{\zeta}\]
 is holomorphic near $x$. By assumption, $q_A-(\d\zeta_k)^2/\zeta_k$ is holomorphic at $x_k$. 
 It follows that $(f^{\circ k})^* q_A - (\d\zeta)^2/\zeta$ is holomorphic near $x$. 

Since $(f^{\circ k})^*\vartheta_v$ is holomorphic near $x$, we therefore have 
\[ \left<\frac{(\d \zeta)^2}{\zeta},(f^{\circ k})^*\vartheta_v\right>_x =  \left<(f^{\circ k})^* q_A,(f^{\circ k})^*\vartheta_v\right>_x = 
\left<q_A,\vartheta_v\right>_{x_k}.\]
As a consequence 
\[\d A(v)  = \left<\frac{(\d \zeta)^2}{\zeta},\sum_{k=0}^{p-1} (f^{\circ k})^*\vartheta_v\right>_x = \sum_{k=0}^{p-1}\left<q_A,\vartheta_v\right>_{x_k} =  \left<q_A,\vartheta_v\right>_{\left<x\right>}.\]
This proves Proposition \ref{prop:dAdB} for $\d A$. The proof for $\d B$ is similar. 
 \end{proof}

\section{Injectivity of $\nabla_f$\label{sec:inj}}

In order to prove Theorem \ref{theo:main}, we need to use the global properties of the map $f$. Up to now, we only used the local properties near the cycle. For this purpose, it is important that the quadratic differentials $q_A$ and $q_B$ which appear in Proposition \ref{prop:dAdB} are globally meromorphic on $\Sp$. Here, we define such quadratic differentials $q_A$ and $q_B$ and we prove that the linear map 
\[\nabla_f := \id-f_*\]
 is well defined and injective on ${\rm Vect}(q_A,q_B)$. 

\subsection{A space of quadratic differentials}
Denote by $\Quad(\TT)$ the space of meromorphic quadratic differentials on $\Sp$ which have at worst simple poles at $z=±{\rm i}\infty$. 
Given $Z\subset \TT$, denote by $\Quad(\TT;Z)\subset \Quad(\TT)$ the subspace of quadratic differentials which are holomorphic outside $Z$. 
Finally, we denote by $\Quad^1(\TT;Z)\subset \Quad(\TT;Z)$ the subspace of quadratic differentials having at worst simple poles. 

\begin{lemma}
Any polar part of quadratic differential along $\left<x\right>$ may be realized as the polar part of a quadratic differential in $\Quad\bigl(\TT;\left<x\right>\cup \{c^+\}\bigr)$ having at worst a simple pole at $c^+$. 
\end{lemma}

\begin{proof}
For all $k\in [0,p-1]$, the quadratic differentials 
\[\frac{(\dw)^2}{\bigl(w-w(x_k)\bigr)\bigl(w-w(c^+)\bigr)w},\quad \frac{(\dw)^2}{\bigl(w-w(x_k)\bigr)^2w}\quad\text{and}\quad \frac{(\dw)^2}{\bigl(w-w(x_k)\bigr)^j}\quad \text{for }j\geq 3\]
belong to $\Quad\bigl(\TT;\left<x\right>\cup \{c^+\}\bigr)$.
The first has a simple pole at $x_k$, the second has a double pole at $x_k$, and the third has a pole of order $j\geq 3$ at $x_k$. Thus, they generate the space of polar parts at $x_k$. 
\end{proof}

From now on, we assume that $q_A\in \Quad\bigl(\TT;\left<x\right>\cup \{c^+\}\bigr)$ and $q_B\in   \Quad\bigl(\TT;\left<x\right>\cup \{c^+\}\bigr)$  have at worst simple poles at $c^+$ and that 
$q_A-(\d\zeta_k)^2/\zeta_k$ and $q_B - (\d\zeta_k)^2/\zeta_k^2$ are holomorphic at $x_k$ for all $k\in [0,p-1]$. 
We set 
\[\qf:={\rm Vect}(q_A,q_B).\] 

\subsection{Pairing quadratic differentials in $\qf$ with vector fields in $\tf$}

Recall that 
\[\tau := w\frac{\d}{\dw} \quad \text{and}\quad \tau - f^*(\tau)\in \tf.\]

\begin{lemma}
For all $q\in \qf$, 
\[\left<q,\tau-f^*\tau\right>_{\left<x\right>}=0.\]
\end{lemma}

\begin{proof}
Assume $q\in \qf$. According to Lemma \ref{lem:invariance}, $q-f^*q$ is holomorphic near $\left<x\right>$. Since $\tau$ is also holomorphic near $\left<x\right>$, and since $f$ is a local isomorphism near $\left<x\right>$, 
\[\left<q,f^*\tau\right>_{\left<x\right>} = \left<f^*q,f^*\tau\right>_{\left<x\right>} = \left<q,\tau\right>_{\left<x\right>}.\qedhere\]
\end{proof}

\subsection{Pushing forward quadratic differentials in $\Quad_f$}

According to Corollary \ref{coro:covering}, $f:\TT\ssm f^{-1}(\vf)\to \TT\ssm \vf$ is a covering map. 
Here, we show that for all $q\in \Quad$, the following series defines a meromorphic quadratic differential on $\Sp$: 
\begin{equation}\label{eq:push}
f_*q  = \sum_{g\text{ inverse branch of }f} g^* q.
\end{equation}
The (minor) difficulty is that the degree of the covering map is not finite, and that $q$ may fail to be integrable on $\Sp$ since it may have multiple poles along $\left<x\right>$. So, we cannot apply directly the results presented in Appendix \ref{sec:qd}. 
The reason why the series in Equation \eqref{eq:push} converges is that $q$ is locally integrable near the essential singularities of $f$, i.e., the points $±{\rm i}\infty$. 

\begin{lemma}
If $q\in \qf$, the series in Equation \eqref{eq:push} converges locally uniformly in $\TT\ssm \bigl(\vf\cup \left<x\right>\bigr)$. 
Its sum $f_*q$ is a meromorphic quadratic differential on $\Sp$. 
\end{lemma}

\begin{proof}
Assume $q\in \qf$. Let $V\subset \TT\ssm \bigl(\vf\cup \left<x\right>\bigr)$ is compactly contained in $\Sp\ssm \left<x\right>$. Then, $U:=f^{-1}(V)$ is compactly contained in $\Sp\ssm \left<x\right>$. In particular, $q$ is integrable on $U$. In addition, $f:U\to V$ is a covering map. It follows that the series in Equation \eqref{eq:push} converges uniformly on $V$ and that $f_*q$ is integrable on $V$. This shows $f_*q$ is holomorphic on $\TT\ssm \bigl(\vf\cup \left<x\right>\bigr)$ and has at worst simple poles at $±{\rm i} \infty$ and on $\vf$. 

To see that $f_*q$ is meromorphic near $x_k$, $k\in [1,p]$,  let $V\subset \TT\ssm \vf$ be a topological disk containing $x_k$. Then, $U:=f^{-1}(V)$ is the disjoint union of a topological disk $U'$ containing $x_{k-1}$ and a open set $U''$ compactly contained in $\Sp\ssm \left<x\right>$. Then, 
$(f|_{U''})_*q$ is holomorphic. In addition, $f:U'\to U$ is an isomorphism so that $ (f|_{U'})_*q $ -- and thus $f_*q = (f|_{U'})_*q + (f|_{U'})_*q$ -- is meromorphic near $x_k$.
\end{proof}

We may therefore consider the linear map
\[\nabla_f: =\id -f_* : \qf \to \Quad(\TT).\]
It will be convenient to set 
\[Y:= \{c^+\}\cup \vf.\]

\begin{lemma}
We have the inclusion 
\[\nabla_f(\qf)\subseteq \Quad^1(\TT;Y).\]
\end{lemma}

\begin{proof}
Assume $q\in \qf$. As mentioned in the proof of the previous lemma, $f_*q$ is holomorphic on $\TT\ssm \bigl(\vf\cup \left<x\right>\bigr)$ and has at worst simple poles at $±{\rm i} \infty$ and on $\vf$. In addition, for $k\in [1,p]$, the polar part of $f_*q$ at $x_k$ is equal to the polar part of $g^* q$ where $g$ is the inverse branch of $f$ sending $x_k$ to $x_{k-1}$. According to Lemma \ref{lem:invariance}, $q-f_*q$ is therefore holomorphic near $\left<x\right>$. It follows that $q-f_*q \in \Quad^1(\TT;Y)$. 
\end{proof}

\subsection{Injectivity of $\nabla_f$}

An observation due to Adam Epstein is that the linear map $\nabla_f$ is injective on $\qf$, and that this is the key to the proof of Theorem \ref{theo:main}. 

\begin{proposition}\label{prop:nabla}
The linear map $\nabla_f : \qf\to \Quad(\TT)$ is injective. 
\end{proposition}

\begin{proof}
We must prove that $f_*q\neq q$ for $q\in \qf\ssm \{0\}$.  
If $q$ were integrable on $\TT$, the result would follow immediately from Proposition \ref{prop:contraction}, since we would have $\|f_*q\|_{L^1(\TT)}<\|q\|_{L^1(\TT)}$. Since $q$ may have double poles near $\left<x\right>$, it may fail to be integrable on $\TT$. In that case, we may proceed as follows. 

Assume $q\in \qf\ssm\{0\}$. For $\eps>0$ small, let $V_\eps\subset \T$ be the union of topological disks
\[V_\eps := \bigcup_{k=1}^p \bigl\{|\zeta_k|<\eps\bigr\}.\]
Set $U_\eps:= f^{-1}(V_\eps)\subset \TT$. Then, $\left<x\right>\subset U_\eps$ and so, $q$ is integrable on $\TT\ssm U_\eps$. As a consequence, 
\[\|f_*q\|_{L^1(\TT\ssm V_\eps)}<\|q\|_{L^1(\TT\ssm U_\eps)}.\]
Similarly, for $\eps'<\eps$, 
\[\|f_*q\|_{L^1(V_\eps\ssm V_{\eps'})}<\|q\|_{L^1(U_\eps\ssm U_{\eps'})}.\]
As a consequence, the function 
\[\eps \mapsto \|q\|_{L^1(\TT\ssm U_\eps)} - \|f_*q\|_{L^1(\TT\ssm V_\eps)}\]
is positive and decreasing. In particular, it has a positive limit. 
Note that 
\[\|q\|_{L^1(\TT\ssm U_\eps)} - \|q\|_{L^1(\TT\ssm V_\eps)} = \|q\|_{L^1(V_\eps\ssm U_\eps)}- \|q\|_{L^1(U_\eps\ssm V_\eps)}\leq  \|q\|_{L^1(V_\eps\ssm U_\eps)} .\]
We deduce from the following lemma that $f_*q\neq q$. 
\end{proof}

\begin{lemma}
For any $q\in \qf$, 
\[\lim_{\eps\to 0} \|q\|_{L^1(V_\eps\ssm U_\eps)}=0.\]
\end{lemma}

\begin{proof}
Since $\zeta\circ f^{\circ p} = \zeta + {\cal O}(\zeta^3)$, there is a constant $\kappa_1$ such that 
\[|\zeta\circ f^{\circ p}|\geq |\zeta|-\kappa_1 |\zeta|^3.\]
Since $q$ has at worst a double pole at $x$, there is a constant $\kappa_2$ such that for $|\zeta|$ small enough
\[|q|\leq \kappa_2\frac{|\d\zeta^2|}{|\zeta|^2}.\]
Note that for $\eps>0$ small enough, 
\[V_\eps\ssm U_\eps = \bigl\{|\zeta|<\eps\bigr\}\ssm \bigl\{|\zeta\circ f^{\circ p}|<\eps\bigr\}\subset \bigl\{\eps-\kappa_1\eps^3\leq |\zeta|< \eps\bigr\}.\]
Thus, 
\[0\leq \|q\|_{L^1(V_\eps\ssm U_\eps)}\leq \int_0^{2\pi} \int_{\eps-\kappa_1\eps^3}^\eps \kappa_2 \frac{r\d r\d t}{r^2} = 2\pi \kappa_2\ln \frac{1}{1-\kappa_1\eps^2}\underset{\eps\to 0}\longrightarrow 0.\qedhere\]
\end{proof}

%

\section{Linear independence\label{sec:trans}}

We may now complete the proof that $\d A|_{\T_\lambda \Lambda}$ and  $\d B|_{\T_\lambda \Lambda}$ are linearly independent.
According to Proposition \ref{prop:dAdB}, for all $v\in \T_\lambda\Lambda$, 
\[\d A(v) = \bigl<q_A,\Theta_f(v)\bigr>_{\left<x\right>}\quad \text{and}\quad \d B(v) = \bigl<q_B,\Theta_f(v)\bigr>_{\left<x\right>}.\]
According to Lemma \ref{lem:tf}, 
\[\tf = {\rm Im}(\Theta_f) \oplus {\rm Vect}(\tau-f^*\tau).\]
Showing that $\d A|_{\T_\lambda \Lambda}$ and  $\d B|_{\T_\lambda \Lambda}$ are linearly independent therefore amounts to proving that for all $q\in \qf\ssm \{0\}$, there exists $\vartheta\in \tf$ such that $\left<q,\vartheta\right>_{\left<x\right>}\neq 0$. 

\subsection{Guiding vector fields}

Set $Z:=\cf\cup \vf$ and denote by $\T Z$ the space of maps $\xi: Z\to \T\TT$ satisfying $\xi(z) \in \T_z \TT$ for all $z\in \Z$. 

\begin{lemma}
For any $\vartheta\in \tf$, there exists a unique $\xi_\vartheta \in \T Z$ such that for any vector field $\xi$, defined and holomorphic near $Z$ with $\xi(z) = \xi_\vartheta(z)$ for $z\in Z$, the vector field $\vartheta +\xi - f^* \xi$ is holomorphic and vanishes along $\cf$. 
\end{lemma}

\begin{proof}
Fix $\vartheta\in \tf$. Let us first prove the uniqueness of $\xi_\vartheta\in \T Z$. Assume $\xi_1$ and $\xi_2$ are two vector fields, defined and holomorphic near $Z$, such that $\vartheta + \xi_1 - f^*  \xi_1$  and $\vartheta + \xi_2 - f^*\xi_2$ are holomorphic near $\cf$. Then, $(\xi_1-\xi_2) - f^*(\xi_1-\xi_2)$ is holomorphic and vanishes along $\cf$. As a consequence, $\D f \circ (\xi_1-\xi_2) - (\xi_1-\xi_2)\circ f$ vanishes on $\cf$. Since $\D f \circ (\xi_1-\xi_2)$ vanishes on $\cf$, this forces $\xi_1-\xi_2$ to vanish on $\vf$. In that case, $f^*(\xi_1-\xi_2)$ vanishes on $\cf$ and so, $\xi_1-\xi_2$ vanishes on $\cf$. This shows the uniqueness of $\xi_\vartheta \in \T Z$. 

This also proves that if $\vartheta +\xi - f^* \xi$ is holomorphic and vanishes along $\cf$ for some vector field $\xi$, defined and holomorphic near $Z$ with $\xi(z) = \xi_\vartheta(z)$ for $z\in Z$, then $\vartheta +\xi - f^* \xi$ is holomorphic and vanishes along $\cf$ for any vector field $\xi$, defined and holomorphic near $Z$ with $\xi(z) = \xi_\vartheta(z)$ for $z\in Z$.

Let us now prove the existence of $\xi_\vartheta\in \T Z$. Note that $\D f\circ \vartheta$ is a map from $\TT\ssm\cf$ to the tangent bundle $\T\TT$. Note that it is not a vector field since for $z\in \TT$, the vector $\D f\circ \vartheta(z)$ belongs to $\T_{f(z)}\TT$. However, since $\vartheta$ has at worst simple poles along $\cf$ and since $\D f$ vanishes on $\cf$, the map $\D f \circ \vartheta$ extends holomorphically to $\TT$.
Set  
\[\xi_\vartheta \bigl( f(c^±)\bigr) := \D f \circ \vartheta(c^±).\]
Next, let $\xi$ be any vector field, defined and holomorphic near $\vf$, coinciding with $\xi_\vartheta$ on $\vf$. 
Then, $\vartheta-f^*\xi$ is holomorphic near $\cf$ and we may set 
\[\xi_\vartheta(c^±) := (\vartheta-f^*\xi) (c^±).\qedhere\]
\end{proof}

Recall that $Y:= \{c^+\}\cup \vf\subset Z$. 
It will be convenient to consider the linear map $\Xi_f:\tf\to \T Y$ defined by 
\[\Xi_f(\vartheta) = \xi_\vartheta|_Y.\]

\begin{lemma}\label{lem:xif}
The map $\Xi_f : \tf\to \T Y$ is an isomorphism. 
\end{lemma}

\begin{proof}
Since the dimensions of $\tf$ and $\T Y$ are both equal to three, it is enough to show that the map is injective. 
Assume $\vartheta \in \tf$ and $\xi_\vartheta$ vanishes on $\{c^+\}\cup \vf$. Let $\xi$ be a vector field, defined and holomorphic near $Z$, which coincides with $\xi_\vartheta$ on $Z$. We may assume that $\xi$ identically vanishes near $\{c^+\}\cup \vf$. Then, 
$\vartheta+ f^*\xi -\xi=\vartheta-\xi$ is holomorphic and vanishes on $\cf$. This shows that $\vartheta$ is holomorphic near $\cf$ and vanishes at $c^+$. As a consequence, $\vartheta$ is globally holomorphic on $\Sp$, and vanishes at three points: $c^+$, $+{\rm i}\infty$ and $-{\rm i}\infty$. So, $\vartheta = 0$. 
\end{proof}

\subsection{From the cycle to the critical set}

We may now transfer the local computations done near the cycle $\left<x\right>$ to local computations done near the critical set $Y$. 

\begin{lemma}
For all $\vartheta\in \tf$ and all $q\in \qf$,
\[\left< q,\vartheta\right>_{\left<x\right>} = \bigl<\nabla_f q,\Xi_f(\vartheta)\bigr>_Y.\]
\end{lemma}

\begin{proof}
Let $\xi$ be a vector field, defined and holomorphic near $Z$, coinciding with $\xi_\vartheta := \Xi_f(\vartheta)$ on $Z$. 
Then, $\vartheta + \xi -f^*\xi$ is holomorphic near $\cf$. 
In addition, since $\nabla_f q$ is holomorphic near $c^-$, 
\begin{align*}
\left<\nabla_f q,\xi_\vartheta\right>_Y = \left<q-f_* q,\xi \right>_Z & = \left<q,\xi \right>_{\cf} - \left<f_* q,\xi \right>_{\vf} \\
&=  \left<q,\xi \right>_\cf - \left<q,f^*\xi \right>_\cf\\
&= \left<q,-\vartheta \right>_\cf  = \left<q,\vartheta \right>_{\left<x\right>}.
\end{align*}
In the second line, we used the fact that the only poles of $q\otimes f^*\xi$  in $f^{-1}(\vf)$ belong to  $\cf$.
For the last equality, we used the fact that $q\otimes \vartheta$ is a globally meromorphic $1$-form on $\Sp$, whose poles are contained in $\cf\cup \left<x\right>$, and that the sum of all residues of a globally meromorphic $1$-form on a compact Riemann surface is $0$. 
\end{proof}

\subsection{Completion of the proof}

Assume by contradiction that  $\d A|_{\T_\lambda \Lambda}$ and  $\d B|_{\T_\lambda \Lambda}$ are not linearly independent. 
Then, there is a $q\in \qf\ssm \{0\}$ such that for all $\vartheta\in \tf$, 
\[0=\left<q,\vartheta\right>_{\left<x\right>}= \bigl<\nabla_f q,\Xi_f(\vartheta)\bigr>_Y.\]
According to  Lemma \ref{lem:xif}, the map $\Xi_f:\tf\to \T Y$ is an isomorphism. In particular, it is surjective. 
It follows that for all $\xi\in \T Y$, 
\[\left<\nabla_f q,\xi\right>_Y=0.\]
As a consequence, $\nabla_f q$ is holomorphic near $Y$ and thus, has at most three simple poles at $c^-$, $+{\rm i}\infty$ and $-{\rm i}\infty$. 
A non zero quadratic differential on $\Sp$ has at least four poles, counting multiplicities. Thus, $\nabla_f q=0$.

According to Proposition \ref{prop:nabla}, the map $\nabla_f:\qf\to \Quad^1(\TT;Y)$ is injective. 
It follows that $q=0$. Contradiction. 

This completes the proof of Theorem \ref{theo:main}. 

\appendix

\section{Quadratic differentials\label{sec:qd}}

\subsection{Meromorphic quadratic differentials}

A  {\em  quadratic  differential} on a Riemann surface $U$ is a section of the square of the cotangent bundle $\T^*U\otimes \T^*U$. We shall usually think of a quadratic differential $q$ as a field of quadratic forms. In particular, if $\vartheta$ is a vector field on $U$ and $\phi$ is a function on $U$, then $q(\vartheta)$ is a function on $U$ and $q(\phi\vartheta) = \phi^2q(\vartheta)$. 

If $\zeta:U\to \C$ is a coordinate, we shall use the notation $(\d\zeta)^2 = \d\zeta\otimes \d\zeta$ - not be confused with $1$-form $\d(\zeta^2)$.  
Then, a quadratic differential $q$ on $U$ is of the form $q = \phi \ (\d\zeta)^2$ for some function $\phi$. We say that $q$ is meromorphic on $U$ if $\phi$ is meromorphic on $U$. In that case, the order of $q$ at a point $x\in U$ is $\ord_x q := \ord_x \phi$, i.e., $0$ if $\phi$ is holomorphic and does not vanish at $x$, $k\geq 1$ if $\phi$ has a zero of multiplicity $k$ at $x$, and $-k\leq -1$ if $\phi$ has a pole of multiplicity $k$ at $x$. 

\subsection{Pullback}
The derivative $\D f : \T U \to \T V$ of a holomorphic map $f:U\to V$ naturally induces a pullback map $f^*$ from quadratic differentials on $V$ to quadratic differentials on $U$: 
\[f^*q := q\circ \D f.\]

\begin{lemma}
If $f:(U,x)\to (V,y)$ is holomorphic  at $x$, and $q$ is meromorphic at $y=f(x)$, then 
\[2+\ord_x (f^*q) = \deg_x f\cdot (2+\ord_y q).\]
\end{lemma}

\begin{proof}
Choose local coordinates $z:(U,x)\to (\C,0)$ and $w:(V,y)\to (\C,0)$ such that $w\circ f = z^k$, with $k:=\deg_x f$. If $q=\phi \ (\dw)^2$, then $f^*q =  \phi\circ f \cdot (k z^{k-1} \dz)^2$. Thus, 
\[2+\ord_x (f^*q) = 2 +  \ord_x(\phi\circ f)+ (2k-2)  = 2k+ k\cdot \ord_y \phi = k\cdot (2+\ord_y q).\qedhere\]
\end{proof}

\subsection{Pushforward for finite degree covering maps}

Assume $f:U\to V$ is a finite degree covering map. 
If $q$ is a quadratic differential on $U$, we define a quadratic differential $f_*q$ on $V$ by
\[f_*q := \sum_{g\text{ inverse branch of }f} g^* q.\]
If $q$ is holomorphic on $U$, then $f_*q$ is holomorphic on $V$. 

\begin{lemma}
Assume $U:=\widehat U\ssm \{x\}$ and $V:=\widehat V\ssm \{y\}$ are punctured disks, 
$f:U\to V$ is a covering map ramifying at $x$ with local degree $\deg_xf$ and $q$ is meromorphic at $x$. Then,
$f_*q$ has at worst a pole at $y$ and 
\[
2+\ord_{y} (f_*q)  \geq \frac{2+\ord_{x} q}{\deg_xf}.
\] 
\end{lemma}

\begin{proof}
The group of deck transformations of $f:U\to V$ is a cyclic group of order $\deg_x f$. Note that 
\[f^*(f_* q) = \sum_{h\text{ deck transformation of }f} h^* q,\]
and $\ord_x h^* q = \ord_x q$ for all deck transformations $h$, so that 
\[\ord_x f^*(f_* q) \geq \ord_x q.\]
Then, 
\[
2+\ord_{y} (f_*q) = \frac{2+\ord_x f^*(f_* q )}{\deg_x f}  \geq \frac{2+\ord_{x} q}{\deg_xf}.\qedhere
\] 
\end{proof}

\subsection{Integrable quadratic differentials}

If $q$ is a quadratic differential on $U$, we denote by $|q|$ the positive $(1,1)$-form on $U$ defined by 
\[|q|(\vartheta_1,\vartheta_2) := \frac{1}{2} \bigl|q(\vartheta_1-{\rm i}\vartheta_2)\bigr| - \frac{1}{2} \bigl|q(\vartheta_1+{\rm i}\vartheta_2)\bigr|.\]
If  $\zeta:U\to \C$ is a coordinate and $q = \phi\ (\d\zeta)^2$, then 
\[|q| = |\phi|\cdot  \frac{\rm i}{2} \d\zeta\wedge \d\bar \zeta.\]
We shall say that $q$ is {\em integrable} on $U$ if 
\[\|q\|_{L^1(U)}:=\int_U |q|<\infty.\]
Note that $q$ is integrable in a neighborhood of a pole if and only if the pole is simple. 
If $f:U\to V$ is an isomorphism and $q$ is an integrable quadratic differential on $V$, then $f^*q$ is integrable on $U$ and 
$\|f^* q\|_{L^1(U)} = \|q\|_{L^1(V)}$.

\subsection{Pushforward for infinite degree covering maps}

Assume $f:U\to V$ is an infinite degree covering map. If $q$ is an integrable quadratic differential on $U$, we may still define
\[f_*q := \sum_{g\text{ inverse branch of }f} g^* q.\]
Indeed, the series converges in $L^1_{\rm loc}$ since if $V'\subset V$ is a topological disk, so that the inverse branches $g:V'\to U$ of $f$ are defined on $V'$, and if $U':= f^{-1}(V')$, then
\[\sum \|g^* q\|_{L^1(V')}  = \|q\|_{L^1(U')}\leq \| q\|_{L^1(U)}.\]
The limit  of a sequence of holomorphic functions converging in $L^1_{\rm loc}$ is itself holomorphic. It follows that if 
$q$ is holomorphic on $U$, then $f_*q$ is holomorphic on $V$. 

\subsection{The Contraction Principle}

\begin{proposition}\label{prop:contraction}
Let $f:U\to V$ be a covering map and let $q$ be a holomorphic integrable quadratic differential on $U$. 
Then, $\|f_*q\|_{L^1(V)}\leq \|q\|_{L^1(U)}$ and equality holds if and only if either $q=0$, or the degree of $f$ is finite and $f^*(f_*q)=\deg (f)\cdot q$. 
\end{proposition}

\begin{proof}
The proof is an immediate application of the triangle inequality: for any topological disk $V'\subset V$, we have 
\[\int_{V'}|f_*q| = \int_{V'}\left|\sum g^*q\right|
\leq \int_{V'}\sum  |g^* q| = \sum \int_{V'} |g^* q| = \int_{f^{-1}(V')} |q|,
\]
where the sums range over the inverse branches $g:V'\to U$ of $f$. 
It follows that 
\[\int_V |f_*q|\leq \int_{f^{-1}(V)}|q| =  \int_{U}|q| \]
with equality if and only if for all inverse branches $g$ of $f$, we have $g^*q=\psi_g\ f_*q$ for some  function $\psi_g:V'\to [0,1]$ satisfying $\sum_g \psi_g=1$. Setting $\phi\bigl(g(y)\bigr):=\psi_g(y)$, we see that $q = \phi \ f^*(f_*q)$ for some function $\phi:U\to [0,1]$. Since $q$ and $f^*(f_* q)$ are holomorphic, either $q=0$, or the function $\phi$ is  constant, let us say equal to $c\in [0,1]$. Since  $\sum_g \psi_g=1$, we have that $\deg(f)\cdot c = 1$, which forces the degree of $f$ to be finite with $f^*(f_*q)=\deg (f)\cdot q$.
\end{proof}

%
%

\subsection{Pairing quadratic differentials and  vector fields\label{subsec:pairing}}

If $q$ is a quadratic differential on $U$ and $\vartheta$ is a vector field on $U$, we may consider the $1$-form $q\otimes \vartheta$ defined on $U$ by its action on vector fields $\tau$:
\[q\otimes\vartheta(\tau) = \frac{1}{4}\bigl(q(\vartheta+\tau)-q(\vartheta-\tau)\bigr).\]
Note that if $q = \phi\ (\d\zeta)^2$ and $\vartheta=\psi\ \d/\d\zeta$, then $ q\otimes \vartheta = \phi \psi \ \d\zeta.$

If $x\in U$, and if $\vartheta$ and $q$ are meromorphic on $U$, we set 
\[\left<q,\vartheta\right>_x:=\res(q\otimes \vartheta,x).\]
If $q$ has at worst a simple pole at $x$, then $\left<q,\vartheta\right>_x$ only depends on $\theta:=\vartheta(0)$, and we use the notation
\[\left<q,\theta\right>_x:=\left<q,\vartheta\right>_x.\]

\begin{lemma}
Let $U:=\widehat U\ssm \{x\}$ and $V:=\widehat V\ssm \{y\}$ be punctured disks, let
$f:U\to V$ be a covering map ramifying at $x$, let $q$ be a meromorphic quadratic differential on $\widehat U$ and let $\vartheta$ be a meromorphic vector field on $U$.
Then 
\[ \left<f_* q, \vartheta\right>_y = \left< q,f^*\vartheta\right>_x.\]
\end{lemma}

\begin{proof}
Let $\gamma\subset V$ be a loop around $y$ with basepoint $a$. Then 
\[
\int_\gamma (f_* q) \otimes \vartheta=\sum_g \int_{\gamma\ssm\{a\}} (g^*q)\otimes \vartheta=\sum_g\int_{g(\gamma\ssm\{a\})} q \otimes f^*\vartheta= \int_{f^{-1}(\gamma)} q\otimes f^*\vartheta,
\]
where the sum ranges over the inverse branches $g$ of $f$ defined on $\gamma\ssm\{a\}$. 
\end{proof}

\end{document}